\documentclass[12pt,a4paper]{amsart}
\usepackage{amsfonts}
\numberwithin{equation}{section}

\theoremstyle{plain}
\newtheorem{Th}{Theorem}[section]
\newtheorem{Lemma}[Th]{Lemma}

 \theoremstyle{definition}

\newtheorem{?}[Th]{Problem}

\newcommand{\im}{\operatorname{im}}

\begin{document}

\title{Pl\"unnecke's inequality for different summands}

\author{Katalin Gyarmati}
\address{Alfr\'ed R\'enyi Institute of Mathematics\\
     Budapest, Pf. 127\\
     H-1364 Hungary
}
\email{gykati@cs.elte.hu}
\thanks{Supported by Hungarian National Foundation for Scientific Research
(OTKA), Grants No. T 43631, T 43623, T 49693. }

\author{M\'at\'e Matolcsi}
\address{Alfr\'ed R\'enyi Institute of Mathematics\\
     Budapest, Pf. 127\\
     H-1364 Hungary\\
     (also at BME Department of Analysis, Budapest, H-1111, Egry J. u. 1)
} \email{matomate@renyi.hu}
\thanks{Supported by Hungarian National Foundation for Scientific Research
(OTKA), Grants No. PF-64061, T-049301, T-047276}

\author{Imre Z. Ruzsa}
\address{Alfr\'ed R\'enyi Institute of Mathematics\\
     Budapest, Pf. 127\\
     H-1364 Hungary
}
\email{ruzsa@renyi.hu}
\email{{\rm To all authors: } triola@renyi.hu}
\thanks{Supported by Hungarian National Foundation for Scientific Research
(OTKA), Grants No. T 43623, T 42750, K 61908.}


    \begin{abstract}
The aim of this paper is to prove a general version of
Pl\"unnecke's inequality. Namely, assume that for finite sets $A$,
$B_1, \dots B_k$ we have information on the size of the sumsets
$A+B_{i_1}+\dots +B_{i_l}$ for all choices of indices $i_1, \dots
i_l.$ Then we prove the existence of a non-empty subset $X$ of $A$
such that we have `good control' over the size of the sumset
$X+B_1+\dots +B_k$. As an application of this result we generalize
an inequality of \cite{gymr} concerning the submultiplicativity of
cardinalities of sumsets.

    \end{abstract}

     \maketitle

     \section{Introduction}

     Pl\"unnecke \cite{plunnecke70} developed a graph-theoretic method to
estimate the density of sumsets $A+B$, where $A$ has a
positive density and $B$ is a basis. The third author published a simplified
version of his
proof \cite{r89e,r90a}. Accounts of this method can be found in
  Malouf
\cite{malouf95}, Nathanson \cite{nathanson96}, Tao and Vu \cite{taovu06}.

     The simplest instance of Pl\"unnecke's inequality for finite sets goes as
follows.

     \begin{Th} \label{plunnalk}
Let $l<k$ be integers, $A$, $B$ sets in a commutative group and
write $|A|=m$, $|A+lB|=\alpha m$. There exists an $X \subset  A$,
$X \ne   \emptyset $ such that
     \begin{equation} \label{pl1}
      |X+kB| \leq \alpha ^{k/l} |X|.
     \end{equation}
     \end{Th}

     Pl\"unnecke deduced his results from certain properties of the graph built
on the sets $A$, $A+B$, \dots , $A+kB$ as vertices (in $k+1$ different copies of
the group), where from an $x\in A+iB$ edges go to each $x+b\in A+(i+1)B$. This
property (which he called ``commutativity'') was based on the possibility of
replacing a path from $x$ to $x+b+b'$ through $x+b$ by a path through $x+b'$,
so commutativity of addition and the fact that we add the same set $B$
repeatedly seemed to be central ingredients of this method. Still, it is
possible to relax these assumptions. Here we concentrate on the second of them.

     In \cite{r89e} the case $l=1$ of Theorem \ref{plunnalk} is extended to the
addition of different sets as follows.

     \begin{Th} \label{pldiff}
Let $A$, $B_1, \dots  , B_k$ be finite sets in a commutative group and write
$|A|=m$, $|A+B_i|=\alpha _i m$, for $1\leq i\leq h$. There exists an $X
\subset A$, $X \ne \emptyset $ such that
\begin{equation}\label{difplu}
|X+B_1+ \dots +B_k| \leq \alpha _1 \alpha _2 \dots  \alpha _k |X| .
\end{equation}
     \end{Th}

     The aim of this paper is to give a similar extension of the general
     case. This extension will then be applied in Section \ref{sec5} to prove a
conjecture from our paper \cite{gymr}.

     \begin{Th} \label{plgen}
     Let  $l<k$ be integers, and let $A$, $B_1, \dots  , B_k$ be finite sets in
a commutative group $G$. Let $K=\{1, 2, \dots , k\}$, and for any
$I\subset K$ put
     $$   B_I = \sum _{i\in I} B_i ,  $$
     \[   \left|A \right|=m, \ \ \left|A+B_I \right| = \alpha _I m . \]
     (This is compatible with the previous notation if we identify a
one-element subset of $K$ with its element.) Write
     \begin{equation} \label{beta}
     \beta  =  \left( \prod _{L\subset K, \left|L \right|=l} \alpha _L \right)^{(l-1)! (k-l)!/(k-1)!} . \end{equation}
      There exists an $X \subset A$, $X \ne \emptyset $ such that
     \begin{equation}\label{plugen}
     |X+B_K| \leq  \beta  |X| .
     \end{equation}
     \end{Th}

     The problem of relaxing the commutativity assumption will be the subject
of another paper. Here we just mention without proof the simplest case.

     \begin{Th}  \label{baljobb}
Let $A$, $B_1, B_2$ be sets in a (typically noncommutative group)  $G$
 and write $|A|=m$, $\left|B_1+A \right|= \alpha _1m$, $|A+B_2|=\alpha _2 m$. There is an
$X \subset  A$, $X \ne  \emptyset $ such that
     \begin{equation} \label{plnc}
     |B_1 + X+B_2| \leq  \alpha _1 \alpha _2  |X| . \end{equation}
     \end{Th}

     The following result gives estimates for the size of the set $X$ in Theorem \ref{plgen} and a
more general property than \eqref{plugen}, but it is weaker by a
constant. We do not make any effort to estimate this constant; an
estimate could be derived from the proof, but we feel it is
probably much weaker than the truth.

     \begin{Th} \label{plgen2}
     Let  $l<k$ be positive integers, and let $A$, $B_1, \dots  , B_k$ be
finite sets in a commutative group $G$.
     Let  $K, B_I, \alpha _I $ and $\beta $ be as in Theorem \ref{plgen}.
     For any $J\subset K$ such that $l<j=\left|J \right|\leq k$ define
     \begin{equation} \label{betadef}
     \beta _J =  \left( \prod _{L\subset J, \left|L \right|=l} \alpha _L \right)^{(l-1)! (j-l)!/(j-1)!} . \end{equation}
     (Observe that $\beta _K=\beta $ of \eqref  {beta}.)
     Let furthermore a number $\varepsilon $ be given, $0<\varepsilon <1$.
      There exists an $X \subset A$, $\left|X \right|>(1-\varepsilon )m$ such that
     \begin{equation}\label{plugen2}
     |X+B_J| \leq  c \beta _J\left|X \right|
     \end{equation}
     for every $J\subset K$, $\left|J \right|\geq l$. Here $c$ is a constant that depends
on $k,l$ and $\varepsilon $.
     \end{Th}

     We return to the problem of finding large subsets in Section \ref{nagy}.

     \section{The case $k=l+1$}

     First we prove the case $k=l+1$ of Theorem \ref{plgen} in a form which is
weaker by a constant.

     \begin{Lemma} \label{foeset}
     Let  $l$ be a positive integer, $k=l+1$,  and let $A$, $B_1, \dots
, B_k$ be finite sets in a commutative group $G$.
     Let  $K, B_I, \alpha _I $  be as in Theorem \ref{plgen}.
 %  Let $K=\{1, 2,
 % \dots , k\}$, and for any $I\subset K$ put
 %      $$   B_I = \sum _{i\in I} B_i ,  $$
 %      \[   \left|A \right|=m, \ \ \left|A+B_I \right| = \alpha _I m . \]
  Write
      \[   \beta  =  \left( \prod _{L\subset K, \left|L \right|=l} \alpha _L \right)^{1/l} . \]
     (Observe that this is the same as $\beta $ of \eqref  {beta} in this particular
case.)
      There exists an $X \subset A$, $X \ne \emptyset $ such that
     \begin{equation}\label{spec}
     |X+B_K| \leq  c_k\beta  |X|
     \end{equation}
     with a constant $c_k$ depending on $k$.
     \end{Lemma}
\begin{proof}
 % The proof is not difficult, although the notations are somewhat
 % cumbersome.
 Let $H_1, \dots H_k$ be cyclic groups of order $n_1,
\dots n_k$, respectively, let $H=H_1\times H_2\times \dots \times
H_k$, and consider the group $G'=G\times H=G\times H_1\times \dots
\times H_k$. Introduce the notation $B_i'=B_i\times \{0\}\times
\dots \times \{0\} \times H_i \times \{0\}\times \dots \times
\{0\}$ which will be abbreviated as $B_i'=B_i\times H_i$, in the
same manner as $A\times \{0\}\times \dots \times\{0\}$ will still
be denoted by $A$.

     We introduce the notation $i^* = K \setminus  \{i\} = \{1, \dots , i-1, i+1, \dots , k\}$
which gives naturally $B_{i^\ast}=\sum_{j\ne i}B_j$ and,
correspondingly, $\alpha_{i^\ast}=\alpha_{\{1,2, \dots, i-1, i+1,
\dots k\}}$. Note that we have $\prod  \alpha_{i^\ast}= \beta ^l $.

Similarly, let $H_{i^\ast}=H_1\times \dots \times  H_{i-1}\times \{0\}\times H_{i+1}\times \dots \times H_k$, and
$B_{i^\ast}'=\sum_{j\ne i}B_i'=B_{i^\ast}\times H_{i^\ast}$.

Let $q$ be a positive integer (which should be thought of as a
large number), and let $n_i=\alpha_{i^*} q$. We restrict $q$ to values for
which these are integers; such values exist, since the numbers $\alpha _L$ are
rational. Then $\left|H \right|=n=\prod  n_i = \beta ^l q^k$ and $\left|H_{i^\ast}
\right| = n/n_i =(\beta q)^l/\alpha_{i^*} $.
Hence
$|A+B_{i^\ast}'| = |A+B_{i^\ast}| \left|H_{i^\ast} \right| = m (\beta q)^l $
 % =\alpha_{i^\ast}m\frac{n}{n_i} = mq^{k-1}\prod_{j=1}^{k}\alpha_{j^\ast}$,
 independently of $i$.
 % qore precisely,
 % $|A+B_{i^\ast}'|=\alpha_{i^\ast}m\frac{n}{n_i}\leq m\frac
 % {q^k}{q-1}\prod_{j=1}^{k}\alpha_{j^\ast}$, for all $i$, because
 % $\frac{\alpha_i}{n_i}\leq \frac{1}{q-1}$.

Now, let $B'=\bigcup _{i=1}^k B_i'$, and consider the cardinality of
the set $A+(k-1)B'$. The point is that the main part of this
cardinality comes from terms where the summands $B_i'$ are all
different, i.e. from terms of the form $A+B_{i^\ast}'$,
$i=1,2,\dots , k$. There are $k$ such terms, so their cardinality
altogether is not greater than
     \begin{equation}\label{fotag} km  (\beta q)^l         . \end{equation}
 The rest of the terms
all contain some equal summands, e.g. $A+B_1'+B_1'+B_2'+B_3'\dots
+B_{k-2}'$, containing two copies of $B_1'$, etc. The number of
such terms is less than $k^k$, and each of them has `small'
cardinality for the simple reason that $H_i+H_i=H_i$. For instance, in
the example above we have $|A+B_1'+B_1'+B_2'+B_3'\dots
+B_{k-2}'|\leq m |B_1|(\prod_{j=1}^{k-2} |B_j|n_j)\leq c(A, B_1,
\dots B_k) q^{k-2}$ where $c(A, B_1, \dots B_k)$ is a constant
depending on the sets $A, B_1, \dots B_k$ but not on $q$.
Therefore the cardinality of the terms containing some equal
summands is not greater than
    \begin{equation}\label{mellektag}
    k^kc(A, B_1, \dots B_k) q^{k-2}=c(k, A, B_1, \dots B_k) q^{k-2} = o(q^l)
    \end{equation}
     Therefore, combining \eqref{fotag}and
\eqref{mellektag} we conclude that
    \begin{equation}\label{osszes}
    |A+(k-1)B'| \leq  2 km  (\beta q)^l \end{equation}
    if $q$ is chosen large enough.

Finally, we apply Theorem \ref{plunnalk} to the sets $A$ and $B'$
in $G'$. We conclude by \eqref{osszes} that there exists a subset
$X\subset A$ such that
     \begin{equation}\label{xes}|X+kB'|\leq |X| \left( 2 k  (\beta q)^l \right)^{k/l} =
     c_k |X| (\beta q)^k. \end{equation}
Also, observe that $X+(B_K\times
H)\subset X+kB'$, and $|X+(B_K\times H)|=n|X+B_K|.$ From these
facts and \eqref{xes} we obtain
     \[   |X+B_K|  \leq   c_k |X| (\beta q)^k/n = c_k \beta  \left|X \right|   \]
as desired.
\end{proof}

     \section{The general case}

     In this section we prove Theorem \ref{plgen2}.

     As a first step we add a bound on $\left|X \right|$ to Lemma \ref{foeset}.

     \begin{Lemma} \label{kezdolepes}
     Let $k=l+1$, and let $A, B_i, B_I, \alpha _I $ and $\beta $ be as in Lemma
\ref{foeset}.
     Let a number $\varepsilon $ be given, $0<\varepsilon <1$.
      There exists an $X \subset A$, $\left|X \right|>(1-\varepsilon )m$ such that
     \begin{equation}\label{g1}
     |X+B_K| \leq  c(k,\varepsilon ) \beta  \left|X \right|
     \end{equation}
     with a constant $c(k,\varepsilon ) = c_k\varepsilon ^{-\frac{k}{k-1}}$ depending on
 $k$ and $\varepsilon $.
     \end{Lemma}

     \begin{proof} % \label{ }
     Take the largest $X\subset A$ for which \eqref  {g1} holds. If $\left|X
\right|>(1-\varepsilon )m$, we are done. Assume this is not the case. Put $A'=A \setminus  X$, and
apply Lemma \ref{foeset} with $A'$ in the place of $A$. We know that $\left|
A'\right| \geq  \varepsilon m$. The assumptions will hold with
     \[   \alpha _I' = \left|A' + B_i \right|/\left|A' \right| \leq   \left|A + B_i \right|/\left|A' \right|
\leq  \alpha _I/\varepsilon  \]
     in the place of $\alpha _I$. We get a nonempty $X'\subset A'$ such that
     \[ |X'+B_K| \leq  c_k \beta ' |X'| \]
     with
     \[   \beta ' =  \left( \prod _{L\subset K, \left|L \right|=l} \alpha _L' \right)^{1/(k-1)} \leq  \beta
\varepsilon ^{-\frac{k}{k-1}}. \]
     Then $X\cup X'$ would be a larger set, a contradiction.
     \end{proof}

     Now we turn to the general case.

     \begin{Lemma} \label{indukcio}
     Let $J_1, \dots , J_n$ be a list of all subsets of $K$ satisfying $l<\left|
J\right|\leq k$ arranged in an increasing order of cardinality (so
that $J_n=K$); within a given cardinality the order of the sets
may be arbitrary.

     Let $A, B_i, B_I, \alpha _I $ and $\beta _I$ be as in Theorem
     \ref{plgen2}, and
     let the numbers $0<\varepsilon <1$ and $1\leq r\leq n$ be given.
      There exists an $X \subset A$, $\left|X \right|>(1-\varepsilon )m$ such that
     \begin{equation}\label{g2}
     |X+B_J| \leq  c(k,l,r,\varepsilon ) \beta _J\left|X \right|
     \end{equation}
     for every $J=J_1, \dots , J_r$ with a constant $c(k,l,r,\varepsilon )$ depending
on $k,l,r$ and $\varepsilon $.
     \end{Lemma}

 Theorem \ref{plgen2} is the case
$r=n$.

     \begin{proof} % \label{ }
     We shall prove the
statement by induction on $r$.
     Since the sets are in increasing order of size, we have $\left|J_1
\right|=l+1$, and the claim for $r=1$ follows from Lemma \ref{kezdolepes}.

     Now assume we know the statement for $r-1$. We apply it with $\varepsilon /2$ in the
place of $\varepsilon $, so we have a set $X\subset A$, $\left|X \right|>(1-\varepsilon /2)m$ such that
 \eqref  {g2} holds for $J= J_1, \dots , J_{r-1}$ with $c(k,l,r-1, \varepsilon /2)$. Write $A'=X$.
 This set satisfies the assumptions with
     \[   \alpha _I' = \alpha _I/(1-\varepsilon /2).  \]
     We have $\left|J_r \right|=k'$ with some $k'$, $l<k'\leq k$. We are going to
apply Lemma \ref{kezdolepes} with $A', k'$ in the place of $A, k$ and $\varepsilon /2$ in
the place of $\varepsilon $. To this end we need bounds for $\left|A'+B_L \right|$ for
every $L$ such that $\left|L \right|=l'=k'-1$. By the inductive assumption we
know
     \[   \left|A'+B_L \right| \leq  c(k,l,r-1, \varepsilon /2) \beta _L \left|A' \right| .\]
     Lemma \ref{kezdolepes} gives us a set $X'\subset A'$ such that
     \[   \left|X' \right| > (1-\varepsilon /2) \left|A' \right| > (1-\varepsilon ) m\]
     and
     \[   \left| X' + B_{J_r} \right|  \leq  c(l', \varepsilon /2) \beta ' \left|X' \right|, \]
     where
     \[   \beta ' =  \left( \prod _{L\subset J_r, \left|L \right|=l'} c(k,l,r-1, \varepsilon /2) \beta _L \right)^{1/l}
= c(k,l,r-1, \varepsilon /2) \beta _{J_r} .    \]
     In the last step we used an identity among the quantities $\beta _J$ which
easily follows from their definition \eqref  {betadef}.

     The desired set $X$ will be this $X'$, and the value of the constant is
     \[ c(k,l,r,\varepsilon ) =  c(l', \varepsilon /2) c(k,l,r-1, \varepsilon /2)  .  \]
     \end{proof}

     \section{Removing the constant}

     In this section we prove Theorem \ref{plgen}. This is done
     with the help of Theorem \ref{plgen2} and
     technique of taking direct powers of the appearing groups,
     sets, and corresponding graphs.

\begin{proof}[Proof of Theorem \ref{plgen}.]
     Consider the following bipartite directed graph $\mathcal{G}^1$. The first collection of
     vertices $V_1$
     are the elements of set $A$, and the second collection of
     vertices $V_2$ are the elements of set $A+B_K$ (taken in two different
copies of the ambient group to make them disjoint).
     There is an
     edge in $\mathcal{G}^1$ from $v_1=a_1\in V_1$ to $v_2=a_2+b_{1,2}+\dots b_{k,2}\in V_2$ if and only if
there exist elements $b_{1,1},\dots b_{k,1}$ such that
$a_1+b_{1,1}+\dots b_{k,1}=a_2+b_{1,2}+\dots b_{k,2}$. The image
of a set $Z\subset V_1$ is the set $\im Z\subset V_2$ reachable
from $Z$ via edges. The {\it magnification ratio} $\gamma$ of the
the graph $\mathcal{G}^1$ is $\min \{ \frac{|\im Z|}{|Z|}, Z\subset
V_1\}.$ The statement of Theorem \ref{plgen} in these terms is
that $\gamma\leq \beta$, with $\beta$ as defined in the theorem.

Consider now the direct power $\mathcal{G}^r=\mathcal{G}^1\times
\mathcal{G}^1\times \dots \times \mathcal{G}^1$ with collections
of edges $V_1^r=V_1\times \dots \times V_1$ and $V_2^r=V_2\times
\dots \times V_2$, and edges from $(v_1^1, v_2^1, \dots , v_r^1)
\in V_1^r$ to $(v_1^2, v_2^2, \dots v_r^2)\in V_2^r$ if and only
if there exist $\mathcal{G}^1$-edges in each of the coordinates.
Observe that the directed graph $\mathcal{G}^r$ corresponds exactly to
the sets $A^r$ and $A^r+(B_1^r+ \dots + B_k^r)$ in the direct
power group $G^r$. Applying Theorem \ref{plgen2} in the group
$G^r$ to the sets $A^r, B_1^r, \dots B_k^r$ with any fixed
$\varepsilon$, say $\varepsilon =1/2$, we obtain that the
magnification ration $\gamma_r$ of $\mathcal{G}^r$ is not greater
than $c\beta^r$. On the other hand, the magnification ratio is
multiplicative (see \cite{r89e} or \cite{nathanson96}), so that we
have $\gamma_r=\gamma^r$. Therefore we conclude that $\gamma\leq
\sqrt[r]{c}\beta$ and, in the limit, $\gamma\leq \beta$ as desired.
\end{proof}

     \section{Finding a large subset} \label{nagy}

     We give an effective version of Theorem \ref{plgen2} in the original case,
that is, when only $X+B_K$ needs to be small.

     \begin{Th}  \label{plgennagy}
     Let  $A, B_i, B_I, \alpha _I $ and $\beta $ be as in Theorem \ref{plgen}.
     \item{(a)}
Let an integer $a$ be given, $1\leq a\leq m$. There exists an $X \subset  A$, $|X|\geq a$ such that
     \begin{equation} \label{nagya}
      |X+B_K| \leq     \end{equation}
      \[     \leq    \beta m^{k/l} \left( m^{-k/l} + (m-1)^{-k/l}+ \dots  +  (m-a+1)^{-k/l}
+ \bigl(\left|X \right|-a\bigr)(m-a+1)^{-k/l} \right) .       \]
     \item{(b)}
Let a real number $t$ be given, $0\leq t<m$. There exists an $X \subset  A$, $|X|>t$ such
that
     \begin{equation} \label{nagyt}
      |X+B_K| \leq    \beta m^{k/l} \left( {l\over k-l} \left( (m-t)^{1-k/l}-
m^{1-k/l}\right)+ \bigl(\left|X \right|-t\bigr)(m-t)^{-k/l} \right) .       \end{equation}
     \end{Th}

     \begin{proof} % \label{}
     To prove (a), we use induction on $a$. The case $a=1$ is Theorem
\ref{plgen}. Now suppose we know it for
     $a$; we prove it for $a+1$. The assumption gives us
a set $X$, $|X|\geq a$ with a bound on $|X+B_K| $ as given by
\eqref{nagya}. We want to find a set $X'$ with $|X'|\geq a+1$ and
     \begin{equation} \label{pld1}
     \left|X'+B_K \right|  \leq  \end{equation}
      \[     \leq    \beta m^{k/l} \left( m^{-k/l} + (m-1)^{-k/l}+ \dots  +  (m-a)^{-k/l}
+ \bigl(\left|X \right|-a-1\bigr)(m-a)^{-k/l} \right) .       \]
 If $|X|\geq a+1$, we can put $X'=X$. If $|X|=a$, we apply Theorem \ref{plgen}
to the sets $A'=A\setminus X$, $B_1$, \dots , $B_k$. In doing this the numbers $\alpha _I$
should be replaced by
     \[   \alpha _I' = {\left|A'+B_I \right|\over \left|A' \right|} \leq  {\left|A+B_I \right|\over \left|A' \right|}
= \alpha _I {m\over m-a}.       \]
 This yields a
set $Y\subset A\setminus X$ such that
     $$   |Y+B_K| \leq  \beta ' |Y|  $$
     with
     \[   \beta ' =
      \left( \prod _{L\subset K, \left|L \right|=l} \alpha _L' \right)^{(l-1)! (k-l)!/(k-1)!} \leq  \beta
\left(m\over m-a \right)^{k/l}        \]
     and we put $X'=X\cup Y$.

     To prove part (b)
we apply \eqref{nagya} with $a=[t]+1$. The right
side of \eqref {nagyt} can be written as
     $  \beta m^{k/l}  \int _0^{|X|} f(x) \,d x$,
     where $ f(x) = (m-x)^{-k/l}$ for $0\leq x\leq t$, and  $ f(x) =
(m-t)^{-k/l}$ for $t<x\leq |X|$. Since $f$ is increasing, the integral is
     $\geq  f(0)+f(1)+ \dots  + f(|X|-1) $.
     This exceeds the right side of \eqref{nagya} by a termwise comparison.
     \end{proof}

     \section{An application to restricted sums}\label{sec5}

     We prove the following result, which was conjectured in \cite{gymr}.

\begin{Th}
Let $A,B_1, \dots B_k$ be finite sets in a commutative group, and $S\subset
B_1+ \dots +B_k$. We have
\begin{equation} \label{altrestsum}
|S+A|^{k}\leq |S|\prod_{i=1}^k|A+B_{1} + \dots  + B_{i-1} + B_{i+1}+ \dots  + B_{k} |\
. \end{equation}
\end{Th}

     Two particular cases were established in \cite{gymr}; the case when $S$ is the
complete sum $B_1+\dots +B_k$, and the case $k=2$. The proof in
the sequel is similar to the proof of the case $k=2$, the main
difference being that we use the above generalized Pl\"unnecke
inequality, while for $k=2$ the original was sufficient.

\begin{proof}
Let us use the notation $|A|=m$, $s=\prod_{i=1}^k|A+B_{1} + \dots
+ B_{i-1} + B_{i+1}+ \dots  + B_{k} |$. Observe that if $|S|\leq
(s/m^k)^{\frac{1}{k-1}}$ then
\begin{equation} \label{skicsi}
|S+A|\leq |S||A|=|S|^{\frac{1}{k}}|S|^{\frac{k-1}{k}}m \leq
(|S|s)^\frac{1}{k}
\end{equation}
and we are done.

If $|S|> (s/m^k)^{\frac{1}{k-1}}$, then we will use Theorem \ref{plgennagy},
part (b) with $l=k-1$. Note that the $\beta $ of this theorem can be expressed by
our $s$ as
     \[   \beta  = s^{1/(k-1)} m^{-k/(k-1)} .  \]

We take $t=m- \left (\frac{s}{|S|^{k-1}}\right )^{1/k}.$
Then there exists a set $X\subset A$ such that $|X|=r>t$ and
\eqref{nagyt} holds. For such an $X$ we have
\begin{equation} \label{s1felso}
     |S+X|\leq |B_K+X| \leq (k-1)s^{\frac{1}{k-1}}\left( (m-t)^{-\frac{1}{k-1}}-m^{-\frac{1}{k-1}} \right )+(r-t) \left ({s\over (m-t)^k}\right )^{\frac{1}{k-1}}
     \end{equation}
     and we add to this the trivial bound
\begin{equation} \label{s2felso}
|S+(A\setminus X)|\leq |S||A\setminus X|=|S|(m-r).
\end{equation}

We conclude that
\begin{equation} \label{sfelso}
|S+A|\leq |S+X|+|S+(A\setminus X)|\leq
(k-1)s^{\frac{1}{k-1}}\left(
(m-t)^{-\frac{1}{k-1}}-m^{-\frac{1}{k-1}} \right )+
\end{equation}
\begin{equation*}(r-t) \left ({s\over (m-t)^k}\right
)^{\frac{1}{k-1}}+|S|((m-t)-(r-t))=ks^{1/k}|S|^{1/k}-(k-1)\left
(\frac{s}{m}\right )^{\frac{1}{k-1}}\leq k(s|S|)^{1/k}
\end{equation*}
This inequality is nearly the required one, except for the factor
 $k$ on the right hand side. We can dispose of this factor as follows (once again, the
method of direct powers). Consider the sets $A'=A^r,$
$B_j'=B_j^r$ ($j=1, \dots k$), and $S'=S^r$ in the $r$'th direct
power of the original group.
 Applying equation \eqref{sfelso}
to $A',$ etc., we obtain
\begin{equation} \label{s'felso}
|S'+A'|\leq k(s'|S'|)^{1/k}.
\end{equation}
Since
 $|S'+A'|=|S+A|^r,$ $s'=s^r$ and
$|S'|=|S|^r$,  we get
\begin{equation} \label{sfelsouj}
|S+A|\leq k^{1/r}(s|S|)^{1/k}.
\end{equation}
Taking the limit as $r\to \infty $ we obtain the desired
inequality
\begin{equation} \label{sfelsovege}
|S+A|\leq (s|S|)^{1/k}.
\end{equation}

\end{proof}

     %\bibliographystyle{amsplain}
     %\bibliography{cimek,cikkeim}

 %      \begin{thebibliography}{9}
 %
 %
 %      \end{thebibliography}

     \end{document}